\documentclass[11pt]{amsart}

\usepackage{graphicx}
\usepackage{latexsym}
\usepackage{amsfonts}
\usepackage{color}
\newif\iftikz  \tikztrue %\tikzfalse

\usepackage{amsthm,amssymb,amstext,graphicx}
\iftikz
\usepackage{tikz,subfigure}
\fi

\usepackage[latin1]{inputenc}
\usepackage{amsfonts}
\usepackage{amsmath, amssymb, amsthm}
\usepackage{verbatim}
\usepackage[T1]{fontenc}
\usepackage{graphicx}
\usepackage{enumerate}

\newtheorem{theorem}{Theorem}

\newtheorem{corollary}{Corollary}
\newtheorem{proposition}{Proposition}
\newtheorem{lemma}{Lemma}
\newtheorem{remark}{Remark}
\theoremstyle{definition}
\newtheorem{definition}{Definition}
\newtheorem{ex}{Example}

\def\N{\mathbb {N}}

\def\Z{\mathbb {Z}}
\def\R{\mathbb {R}}
\def\C{\mathbb {C}}

\def\1{{{\mathit 1} \!\!\>\!\! I} }
\def\ie{{i.e.,\ }}
\def\eg{{e.g.~}}
\def\mm{{\mu^\flat}}

\begin{document}

\title{Natural extensions for piecewise affine maps via Hofbauer towers}
\author{Henk Bruin and Charlene Kalle}
\thanks{The second author was supported by the research grant FWF S6913 and the NWO Veni grant 693.031.140. The research was partly supported by the EU FP6 Marie Curie Research Training Network CODY (MRTN 2006 035651).}

\date{Version of \today}


\begin{abstract}
We use canonical Markov extensions (Hofbauer towers)
to give an explicit construction of the natural extensions of various measure preserving endomorphisms, and present some applications to particular examples.
\end{abstract}

\maketitle

\section{Introduction}\label{Intro}
A measure theoretical dynamical system is a quadruple $(X, {\mathcal B}, \mu, T)$, where $(X, {\mathcal B}, \mu)$ is a probability space and $T: X \to X$ is a transformation that preserves the measure $\mu$, \ie $\mu(T^{-1}A)=\mu(A)$ for each set $A \in \mathcal B$. To study the properties of a non-invertible transformation $T$, one can use a natural extension. This is a bigger, invertible system $(Y, {\mathcal C}, \nu, S)$ that preserves both the original dynamics and the measure structure with $\mathcal C$ being the coarsest $\sigma$-algebra that makes this possible. Many properties of a natural extension carry over to the original system. For example, the measure theoretical entropies of both systems are equal and they have the same mixing properties. In \cite{Roh64}, Rohlin gave a canonical construction of a natural extension for a wide class of dynamical systems on Lebesgue spaces. He showed that any two natural extensions of the same system are isomorphic, hence we can speak of {\bf the} natural extension. Different versions however can have their own advantages. 

As a basic example, consider the angle doubling map $T_2 x = 2x \pmod 1$ 
on the unit interval $[0,1)$, preserving Lebesgue measure. 
A geometric version of the natural extension is given by the Baker transformation (Figure~\ref{f:baker}):
\[ B:[0,1)^2 \to [0,1)^2, \qquad (x,y) \mapsto  \big(2x \,(\text{mod }1), (y + \lfloor 2x \rfloor )/2\big).\]
Here the dimension of the space $Y$ is one larger than the dimension of $X$, giving room to separate preimage branches. We can recover the original system simply by projecting onto the first coordinate.

\begin{figure}[ht]
\centering
\includegraphics{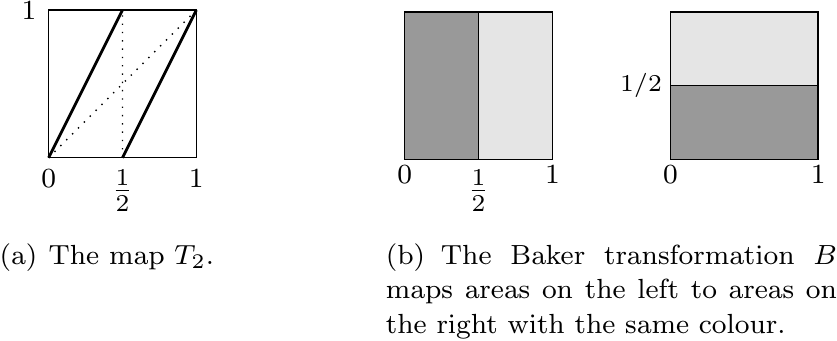}
\caption{The angle doubling map and a geometric version of its natural extension.}
\label{f:baker}
\end{figure}

Such a geometric version of the natural extension can be used for various purposes. For example, for 
$\beta$-transformations (\ie $x \mapsto \beta x \pmod 1$)
it yields an explicit expression of the density of the absolutely continuous invariant measure, see \cite{DKS,DK09}. 
In case $\beta$ is a Pisot number, certain geometric representations of algebraic natural extensions serve to 
identify periodic points (see for example \cite{Aki02,IR06}) and are associated to multiple tilings of a Euclidean space (see for example \cite{KS11,Sch00}). For the standard continued fraction transformation 
$x \mapsto \frac1x - \lfloor \frac1x \rfloor$, also called the Gauss map, a geometric natural extension substantially simplified proofs of results on the quality of the continued fraction approximation coefficients, such as the Doeblin-Lenstra Conjecture and generalisations of Borel's Theorem (see \cite{Jag86,JK89}). For the $\alpha$-continued fraction map $T_\alpha:[\alpha-1, \alpha) \to [\alpha-1, \alpha)$,
$x \mapsto |\frac1x| - \lfloor |\frac1x| + 1-\alpha \rfloor$ for parameter $\alpha \in [0,1]$ a geometric version was recently used to study the behaviour of the entropy as a function of $\alpha$ in \cite{KSS10}.

In this paper we present a general method for obtaining geometric natural extensions of piecewise continuous maps with locally constant Jacobian $J(x) = \frac{d\mu \circ T}{d\mu}(x)$, which generalises the piecewise linearity of some of the above examples. The construction is based on the ``canonical Markov extension'' approach introduced by Hofbauer in \cite{Hof80}, commonly called {\em Hofbauer tower}; see also investigations by Buzzi \cite{Buz99} and Bruin \cite{Bru95}. 
In short, we apply the above ``extend the space with one dimension'' approach that works for the doubling map to the Hofbauer tower. This can be found in Section 2. In Section 3 we show that this natural extension is isomorphic to a countable state Markov shift and that it has several induced transformations that are Bernoulli.
%Since the Hofbauer tower captures (under mild conditions almost all dynamics of the natural extension of the original system  $(X, {\mathcal B}, \mu, T)$, we obtain a geometric model of  $(Y, {\mathcal C}, \nu, S)$ as well. In some cases, it gives us a $T$-invariant measure $\mu$ equivalent to a prescribed reference measure, as well as the natural extension.
In the last two sections we give examples of systems to which the construction applies. These include all piecewise linear expanding interval maps with positive entropy. Other examples are certain higher dimensional piecewise affine maps, in particular a specific skew-product transformation called the random $\beta$-transformation, and rational maps on their Julia set.

\section{The Construction}\label{sec:constr}

In this section we give the construction of the Hofbauer tower and of the geometric natural extension for the class of maps we consider in this article. We first describe this class of transformations.

\subsection{The class of transformations}\label{ss:1to5}
Let $X$ be a compact subset of $\mathbb R^n$, $\mathcal B$ the Lebesgue $\sigma$-algebra on $X$ and $\mu$ a probability measure on $(X,\mathcal B)$. Let $\mathcal Z = \{ Z_j \}_{1 \le j \le N}$ be a collection of closed sets giving a partition of $X$, so $\mu(Z_j)>0$ for all $j$, $\mu(Z_i \cap Z_j) = 0$ for $i \neq j$ and $\mu\big(\cup_{1 \le j \le N} Z_j \big)=1$. Let $T:X \to X$ satisfy the following conditions.
\begin{itemize}
\item[(c1)] For each $Z\in \mathcal Z$, the map $T$ can be extended uniquely to a continuous injective map $T_Z:Z\to \overline{T( \text{int }Z)}$, where $\text{int }Z$ denotes the interior of $Z$ and the bar denotes the closure.
\item[(c2)] For each set $A \in \mathcal B$, also $TA, T^{-1}A \in \mathcal B$ and if $\mu(A)=0$, then also $\mu(T^{-1}A)=0$.
\item[(c3)] The partition $\mathcal Z$ generates $\mathcal B$ in forward time. In other words, $\bigvee_{k \ge 0}T^{-k}\mathcal Z=\mathcal B$, where $\bigvee_{k \ge 0}T^{-k}\mathcal Z$ denotes the smallest $\sigma$-algebra containing all {\em cylinder sets}, \ie the elements of common refinements ${\mathcal Z}_n := \bigvee_{k=0}^n T^{-k}\mathcal Z$.
\end{itemize}
Thus we assume in (c2) that $\mu$ is non-singular w.r.t.\ $T$, but not yet that $\mu$ is $T$-invariant. This assumption will either be made later, or, starting from a reference measure $\mm$, our construction will produce a $T$-invariant measure $\mu \ll \mm$ for which a geometric natural extensions will be constructed. The important step is that we acquire a Markov measure for the Hofbauer tower, which we explain in the following section.

\subsection{The Hofbauer tower}\label{ss:Hofbauer}
Recall that $\mathcal Z_n = \bigvee_{k=0}^n T^{-k}\mathcal Z$ denotes the collection of $(n+1)$-cylinder sets 
$Z_{j_0 \cdots j_n}$, defined by
\[ 
Z_{j_0 \cdots j_n} = Z_{j_0} \cap T^{-1} Z_{j_1} \cap \cdots \cap T^{-n}Z_{j_n},
\]
whenever $\mu(Z_{j_0 \cdots j_n})>0$. Hence $\mathcal Z_0=\mathcal Z$. To obtain the Hofbauer tower we consider the $n$-th images under $T$ of the $(n+1)$-cylinder sets and order them in a convenient way. Indeed, consider the closures of the sets 
\[ T^n Z_{j_0 \cdots j_n}  \quad \text{ for } \quad Z_{j_0 \cdots j_n} \in \mathcal Z_n, \, n \ge 0,\]
with the equivalence relation $\sim$ given by $T^n Z_{j_0 \cdots j_n} \sim T^m Z_{i_0 \cdots i_m}$ if the measure of the symmetric difference
$\mu(T^n Z_{j_0 \cdots j_n} \triangle T^m Z_{i_0 \cdots i_m}) = 0$. Let $\mathcal D$ denote the set of equivalence classes under this relation.
We will occasionally abuse notation and consider the elements of $\mathcal D$ as subsets of $X$ instead of equivalence classes. Note that $T^n Z_{j_0 \cdots j_n} \subseteq Z_{j_n}$.

Clearly $\mathcal D$ is finite or countably infinite, so we can take an ordered index set $\alpha \subseteq \N = \{ 1,2,3,\dots\}$ and write ${\mathcal D} = \{ D_u : u \in \alpha\}$. It is convenient to set $D_u = Z_u$ for $u = 1, \dots, N$, so that the first $N$ elements of $\mathcal D$ are simply the elements of $\mathcal Z$, and we call $\hat X_0 = \sqcup_{u=1}^N D_u$ the {\em base of the Hofbauer tower}. The full {\em Hofbauer tower} $\hat X$ (see \cite{Hof80}) is the disjoint union of the elements of $\mathcal D$,
\[ 
\hat X = \bigsqcup_{n \geq 0}  \bigsqcup_{j_0\cdots j_n} T^n Z_{j_0 \cdots j_n}  / \! \sim \ \ = \bigsqcup_{u \in \alpha} D_u.
\]
\begin{remark}{\rm
There is a choice to define the levels $D_u$ as images $T(D_v \cap Z_j)$ as done by Hofbauer \cite{Hof80} and Keller \cite{Kel89} or as partition elements $T D_v \cap Z_j$ restricted to levels, as is done by Buzzi, \eg \cite{Buz95}. This difference has no profound effect on the outcome; however we follow Buzzi here, as it makes it easier to interpret the dynamics on $\hat X$ as a one-sided subshift 
of $(\alpha^{\N}, \sigma)$.}
\end{remark}

When it is important to specify which component a point $\hat x$ in the Hofbauer tower belongs to, we write $\hat x = (x, D)$ or $(x,u)$ when $D = D_u$. The canonical projection $\pi:\hat X \to X$, $\hat x = (x,D) \mapsto x$, maps the Hofbauer tower onto $X$. Note that $\hat {\mathcal B} := {\mathcal D} \vee \pi^{-1}(\mathcal B)$ is the Lebesgue $\sigma$-algebra on $\hat X$.

\vskip .2cm

We extend the dynamics of $T$ to $\hat X$. Let $D = T^n Z_{j_0 \cdots j_n}  \in \mathcal D$. Then, for each $1 \le j \le N$ such that $\mu(Z_j \cap TD) > 0$, also $D' := Z_j \cap TD=  T^{n+1} Z_{j_0 \cdots j_n j} \in \mathcal D$. Define $\hat T : \hat X \to \hat X$ by
\[ \hat x = (x,D) \mapsto (T x, Z_j \cap TD) \, \text{ if } \, Tx \in Z_j, \] 
and write an arrow $D \to D'$ if this happens. By construction, $\mathcal D$ is a Markov partition of $(\hat X, \hat T)$, and $\pi \circ \hat T = T \circ \pi$. The arrow relation on $(\hat X, \hat T)$ gives rise to a {\em canonical Markov graph} $(\mathcal D, \to)$. Define the symbol space 
\begin{equation}\label{eq:MarkovShift}
\Sigma := \{ y = (y_0y_1y_2\dots) : y_i \in \alpha \text{ and } D_{y_i} \to D_{y_{i+1}} \text{ for all } i \ge 0\}
\end{equation}
indicating all the one-sided paths on $(\mathcal D, \to)$ and let $\sigma: \Sigma \to \Sigma$ denote the left shift, \ie $(\sigma y)_i = y_{i+1}$.  
Let $\eta: \hat X \to \Sigma, \, \hat x \mapsto y$ be given by $y_i = u$ if $\hat T^i \hat x \in D_u$. The system $(\Sigma, \sigma)$ is a factor of $(\hat X, \hat T)$ with {\em factor map} $\eta$, \ie $\eta$ is surjective and $\eta \circ \hat T = \sigma \circ \eta$. A probability measure $\hat \mu$ on $(\mathcal D, \to)$ is called a {\em Markov measure} with transition probabilities $p_{u,v}$ if for all $u,v \in \alpha$
\begin{itemize}
\item $p_{u,v} \in [0,1]$, $p_{u,v}=0$ when $D_u \not \to D_v$ and $\sum_{v:D_u \to D_v} p_{u,v} = 1$, \\
\item $\sum_{u:D_u \to D_v} p_{u,v} \hat \mu(D_u) = \hat \mu(D_v)$.
\end{itemize}
We can extend the Markov measure $\hat \mu$ to ${\mathcal D} \vee \hat T^{-1}{\mathcal D}$ by defining $\hat \mu(D_u \cap \hat T^{-1}D_v) = p_{u,v} \hat \mu(D_u)$. Repeating this to cylinder sets of any length, and extending to the $\sigma$-algebra $\bigvee_{k \ge 0} \hat T^{-k}\mathcal D$, we automatically get that $\hat \mu$ is $\hat T$-invariant. We can take such Markov measures as starting point and define $\mu$ on $X$ as $\mu = \hat \mu \circ \pi^{-1}$.

\begin{lemma} For any Markov measure $\hat \mu$, the projected measure
$\mu = \hat \mu \circ \pi^{-1}$ is $T$-invariant and satisfies conditions (c1)-(c3) of Section~\ref{ss:1to5}.
\end{lemma}

\begin{proof}
Conditions (c1), (c3) and the first part of (c2) do not mention a measure and are just part of the set-up. For the remaining part of condition (c2) we first show that $\bigvee_{k\ge 0} \hat T^{-k}\mathcal D = \mathcal D \vee \pi^{-1}(\mathcal B)$. Recall that $\mathcal B = \bigvee_{k \ge 0} T^{-k}\mathcal Z$. Hence,
\[ \mathcal D \vee \pi^{-1}(\mathcal B) = \bigvee_{k \ge 0} \mathcal D \vee \pi^{-1}(T^{-k}\mathcal Z).\]
To show that $\mathcal D \vee \pi^{-1}(\mathcal B) \subseteq \bigvee_{k\ge 0} \hat T^{-k}\mathcal D$, take any cylinder $Z_{j_0 \cdots j_n}$ and suppose that $\hat \mu\big(\pi^{-1}(Z_{j_0 \cdots j_n})\cap D\big) >0$ for some $D \in \mathcal D$. Then there is a set $D' \in \mathcal D$ such that $D' = \hat T^n\big(\pi^{-1}(Z_{j_0 \cdots j_n})\cap D\big)$ and hence $\pi^{-1}(Z_{j_0 \cdots j_n})\cap D \in \bigvee_{k=0}^n \hat T^{-k} \mathcal D$. The inclusion then follows since $\mathcal D \vee \pi^{-1}(\mathcal B)$ is the smallest $\sigma$-algebra containing all sets of the form $\pi^{-1}(Z_{j_0 \cdots j_n})\cap D$. For the other inclusion, take a non-empty set of the form $D_{u_0} \cap \hat T^{-1} D_{u_1} \cap \cdots \cap \hat T^{-k} D_{u_k}$. This means that there exists a cylinder set $Z=Z_{j_0 \cdots j_n j_{n+1}\cdots j_k}$ such that $T^n Z \subseteq D_{u_0}$ and  $T^{n+k} Z = D_{u_k}$. By the first part of (c2) $T^n Z \in \mathcal B$, so
\[ D_{u_0} \cap \hat T^{-1} D_{u_1} \cap \cdots \cap \hat T^{-k} D_{u_k} = D_{u_0} \cap \pi^{-1} (T^n Z) \in \mathcal D \vee \pi^{-1}(\mathcal B).\] 
Hence, the two $\sigma$-algebras are equal. The $T$-invariance of $\mu$ then follows since $\hat \mu$ is $\hat T$-invariant and $T \circ \pi = \pi \circ \hat T$.  
\end{proof}

\vskip .3cm
\begin{ex}{\rm One example, usually given for finite graphs, but valid for infinite graphs as well provided they are positive recurrent and hence the eigenvectors mentioned below belong to $\ell^2$ (see Gurevi\v{c} \cite{Gur69}), is the {\em Parry measure}, see \cite[Section 8.3]{Walters}. To construct this measure, we assume for simplicity that the graph $(\mathcal D, \to)$ is primitive, and we let $A = (a_{t,u})_{t,u \in \alpha}$ be its adjacency matrix given by $a_{t,u} = 1$ if $D_t \to D_u$ and $a_{t,u} = 0$ otherwise. Let $\lambda$ be the leading eigenvalue; by the Perron-Frobenius Theorem $\lambda > 0$ and its associated left eigenvector $\overline{v} = (\overline v_u)_{u \in \alpha}$ and right eigenvector $\overline{w} = (\overline w_u)_{u \in \alpha}$ can be taken strictly positive. We can scale $\overline v$ and $\overline w$ such that $\sum_{u \in \alpha} \overline v_u \overline w_u = 1$, and construct a stochastic matrix
\[ 
P = (p_{t,u} ), \qquad p_{t,u} = \frac{a_{t,u} \overline v_u}{\lambda \overline v_t}.\]
Finally, the Markov  measure $\hat \mu(D_u) = \overline v_u \overline w_u$ for all $u \in \alpha$, and in general,
\[
\hat \mu(\{ \hat x \ : \ \hat T^k(\hat x) \in D_{u_k}, \ 0 \le k < n \})
= 
\overline v_{u_0} \overline w_{u_0} p_{u_0, u_1} \cdots p_{u_{n-2}, u_{n-1}}, 
\]
is called the {\em Parry measure}. Extended to the $\sigma$-algebra generated by the cylinder sets $\bigvee_{k = 0}^n \hat T^{-k}\mathcal D$, it becomes the measure of maximal entropy of $(\hat X, \hat T)$,
see \cite[Chapter 4.4]{KH}.}
\end{ex}

\subsection{Lifting measures to the Hofbauer tower}\label{ss:hofbauertower}
The above shows that the Markov structure of the Hofbauer Tower always gives a measure on the tower. In general we often have a measure $\mu$ on $(X, \mathcal B)$ that behaves nicely with respect to the map $T$. We would like to determine if there exists a $\hat T$-invariant measure $\hat \mu$ on $\hat X$ such that $\hat \mu \circ \pi^{-1}$ has some relation to $\mu$. Below we follow two strategies of constructing such a measure $\hat \mu$, one in case $\mu$ is $T$-invariant and one in case $\mu$ is not.

Assume that we have a system $(X, \mathcal B, \mu, T)$ satisfying (c1),(c2) and (c3). First extend the measure $\mu$ to a measure $\bar \mu$ on $\hat X$ by setting
\begin{equation}\label{eq:mu_bar}
\bar \mu(A) = \sum_{D \in \mathcal D} (\mu \circ \pi)(A\cap D)
\end{equation} 
for all $A \in \hat{\mathcal B}$. Note that $\bar \mu$ is not (necessarily) $\hat T$-invariant, and can in principle be infinite albeit $\sigma$-finite. Since we have assumed that $\mu(Z)>0$ for each cylinder $Z$, we have $\bar \mu(D) >0$ for all $D \in \mathcal D$. Define a sequence of Cesaro means $(\hat \mu_n)_{n \geq 1}$ on $\hat X$ by setting
\begin{equation}\label{q:mun}
\hat \mu_n(A) = \frac1n \sum_{k=0}^{n-1} \bar \mu(\hat T^{-k}A \cap \hat X_0).
\end{equation}
Here the intersection with base $\hat X_0$ guarantees that $\hat \mu_n$ are all probability measures. The measure $\mu$ is called {\em liftable} if the sequence $\{ \hat \mu_n\}_{n \ge 1}$ from \eqref{q:mun} converges in the vague topology (\ie weak topology on compacta\footnote{Recall that the space $X$ and also the levels $D$ are compact, so $(\hat \mu_n|_D)_n$ has a weak accumulation point for each $D$}) to a non-zero measure $\hat \mu$. Conditions under which measures are liftable are extensively studied, see for example \cite{Kel89,BT07,Buz99,Pes08}. The main point is that there can be no accumulation of mass on the boundaries of sets in the Hofbauer tower and mass cannot escape to infinity. Fix $D_u \in \mathcal D$ and let
\begin{equation}\label{q:partialdn}
\partial_n D_u =  \big\{ Z \in \mathcal Z_n \, : \, 0 < \mu (Z \cap D_u ) < \mu (Z) \big \}.
\end{equation}
In words, $\partial_n D_u$ contains all $(n+1)$-cylinders $Z$ such that $Z$ and $D_u$ have a non-trivial intersection and the cylinder $Z$ is not completely contained in $D_u$. The {\em capacity} of the map $T$ is defined by
\begin{equation}\label{q:cap}
\text{cap}(T) = \limsup_{n \to \infty} \frac1n \log \sup_{u \in \alpha} \#(\partial_n D_u).
\end{equation}
For the proof of Proposition~\ref{p:kel2} and for later use, define the sets
\begin{equation}\label{q:bun}
B_{u,n} = \bigcup_{Z \in \partial_n D_u} (Z\cap D_u).
\end{equation}
We use the notation $\partial A$ for the usual {\em boundary} of a set $A \subset \mathbb R^n$. For the liftability of $\mu$ and the construction of the natural extension in the next section we need to make three additional assumptions on our system.
\begin{itemize}
\item[(c4)]  For each $1 \le j \le N$ there is a constant $s_j \geq 1$ such that for all measurable sets $A \subseteq Z_j$, $\mu(TA) = s_j \mu(A)$.
\item[(c5)] $\mu$ is {\em ergodic}, \ie if $T^{-1}A=A$ for some $A \in \mathcal B$, then $\mu(A) = 0$ or 1.
\item[(c6)] $\mu\big(\cup_n T^n(\cup_{j=1}^N \partial Z_j)\big)=0$.%$\hat \mu(\cup_{u \in \alpha} B_u) = 0$.
\end{itemize}

\begin{remark}{\rm
(i) Condition (c4) requires that the Jacobian of $T$ (see \cite{Par69}) is locally constant; thus $J_{\mu,T}:=\frac{d\mu \circ T}{d\mu}$ has zero distortion.\\
(ii) In (c5) we assume ergodicity without insisting on $T$-invariance. 
Ergodicity of $\mu$ implies that each limit point of $\{ \hat \mu_n\}_{n \ge 1}$ is either zero, or a probability measure.
}
\end{remark}
The next proposition gives some first properties of limit points of the sequence in (\ref{q:mun}).
\begin{proposition}\label{p:erg}
Let $\hat \mu$ be a limit point of the sequence $\{ \hat \mu_n\}$ defined in (\ref{q:mun}). Then $\hat \mu$ is $\hat T$-invariant and ergodic. Also, $\hat \mu \circ \pi^{-1}$ is ergodic.
\end{proposition}

\begin{proof}
The $\hat T$-invariance of $\hat \mu$ follows since it is a limit of Cesaro means. For ergodicity, let $\hat U \subseteq \hat X$ be a measurable set such that $\hat T^{-1}\hat U = \hat U$. Write $U = \pi (\hat U)$. Then $T^{-1}U = U$, so by (c5) either $\mu(U)=0$ or $\mu(X \backslash U)=0$. If $\mu(U)=0$, then
\[ \bar \mu(\hat T^{-k} (\hat U) \cap \hat X_0) = (\mu \circ \pi)\big( \hat T^{-k} (\hat U) \cap \hat X_0 \big) \le  (\mu \circ \pi) (\hat T^{-k}\hat U) \le \mu(T^{-k}U) =0\]
for each $k \ge 0$. Hence, $\hat \mu_n (\hat U)=0$ for all $n$ and so $\hat \mu(\hat U)=0$. Similarly, if $\mu(X \backslash U)=0$, then $\hat \mu (\hat X \backslash \hat U)=0$. Hence $\hat \mu$ is ergodic.

For the last part, let $U \subseteq X$ be a measurable and $T$-invariant set. Then $\pi^{-1} (T^{-1}U) = \hat T^{-1} \pi^{-1}(U)$ and by the previous, $(\hat \mu \circ \pi)(U)$ is either $0$ or $1$.  
\end{proof}

\begin{proposition}[Theorem 2 from \cite{Kel89}]\label{p:kel2}
Assume that $(X, \mathcal B, \mu, T)$ satisfies (c1)-(c6). For a $T$-invariant measure $\mu$ the sequence $\{ \hat \mu_n\}_{n \ge 0}$ converges and if this limit $\hat \mu \not \equiv 0$, then $\hat \mu$ is an ergodic probability measure and $\hat \mu \circ \pi^{-1}=\mu$.
\end{proposition}

\begin{proof}
These results follow from Theorem 2 from \cite{Kel89} by Keller, so we only need to check that the conditions of that theorem are satisfied: $T$ needs to be invariant and ergodic and there has to be a $\mu$-null set $N \subset X$ such that $\hat N = \pi^{-1}N$ has the properties
\begin{itemize}
\item[(2.2)] $\pi^{-1}(A) \in \hat{\mathcal B} \, (\text{mod }\hat N) \Rightarrow A \in \mathcal B \, (\text{mod }\mu)$ for all $A \subseteq X$,
\item[(2.3)] $\hat x, \hat y \in \hat X \setminus \hat N$ and $\pi^{-1}(\hat x) = \pi^{-1}(\hat y)$ imply that $\exists n \ge 0$ s.t.~$\hat T^n \hat x = \hat T^n \hat y$.
\end{itemize}
The ergodicity of $T$ is (c5) and the $T$-invariance is assumed in the proposition. Property (2.2) is satisfied since $\hat{\mathcal B} = \mathcal D \vee \pi^{-1}(\mathcal B)$. Recall the definition of the sets $B_{u,n}$ from (\ref{q:bun}) and set $B_u = \cap_{n \ge 1} B_{u,n}$. Property (2.3) follows from (c6) when we take $N = \pi (\cup_{u \in \alpha} B_u)$, since this implies that the points $\hat x = (x,D_u)$ and $\hat y = (y,D_v)$ are not at the boundary of $D_u$ and $D_v$ respectively. Hence, there is some $n$ and some cylinder $Z_n$, such that $x,y \in Z_n$ and $Z_n$ is contained in the interior of $D_u$ and $D_v$ and this implies that $\hat T^n \hat x = \hat T^n \hat y$. This establishes the existence of a unique vague limit $\hat \mu$. If $\hat \mu \not \equiv 0$, then Theorem 2 from \cite{Kel89} gives the rest of the statement: $\hat \mu \circ \pi^{-1}=\mu$ and $\hat \mu$ is ergodic.  
\end{proof}

\vskip .3cm
\noindent Theorem 3 from \cite{Kel89} by Keller gives conditions under which $\hat \mu \not \equiv 0$ in case of $T$-invariance.
\begin{theorem}[Theorem 3, \cite{Kel89}]\label{thm:Keller3}
Assume that $(X, \mathcal B, \mu, T)$ satisfies (c1)-(c6) and that $\mu$ is $T$-invariant. If $h_{\mu}(T) > \text{cap}(T)$, where $h_{\mu}(T)$ denotes the metric entropy, then the sequence $\{ \hat \mu_n\}_{n \ge 1}$ converges to an ergodic $\hat T$-invariant probability measure $\hat \mu$ for which $\hat \mu \circ \pi^{-1}=\mu$. Moreover, $h_{\hat \mu}(\hat T)=h_{\mu}(T)$.
\end{theorem}

\begin{proof}
Note that $T$-invariance of $\mu$ implies condition (c2). The result by Keller is then valid under (c1), (c3), (c5) and (c6). 
\end{proof}

\vskip .3cm
Invariance of $\mu$ is essential in Theorem~\ref{thm:Keller3} because otherwise $h_\mu(T)$ is undefined, and $\hat \mu \circ \pi^{-1} = \mu$ will fail. However, Theorem~\ref{thm:Keller3} has a version which applies to measures $\mu$ that are non-singular but not necessarily $T$-invariant, as long as (c2) holds. This is due to Keller \cite[Theorem 3(a)]{Kel90} for piecewise smooth interval maps, see also \cite{dMvS}, and \cite{BT07} for the setting of complex polynomials. 
We give one more example for piecewise affine
and expanding maps in $\R^q$. However, it seems fair to say
that proving liftability is not easier than proving the existence
of an invariant measure equivalent to Lebesgue.

\begin{proposition}\label{prop:noninvariant}
Let $X \subset \R^q$ be compact and assume that $T:X \to X$ is piecewise affine and expanding w.r.t.~a finite partition $\mathcal Z$ such that each
$Z \in \mathcal Z$ is a polytope bounded by $(q-1)$-dimensional hyperplanes. Then Lebesgue measure $m^q$ lifts to the Hofbauer tower.
\end{proposition}

\begin{proof}
Tsujii \cite{Tsujii} proved that piecewise affine expanding maps as above have an absolutely continuous invariant probability measure $\mu$ with bounded density $h = \frac{d\mu}{dx}$.
Moreover, there are only finitely many Lebesgue ergodic components (only one if $T$ is transitive), so by passing to a component, we can assume that $q$-dimensional Lebesgue measure $m^q$ is ergodic.

In short, there is no need to use the Hofbauer tower approach to find $\mu$. We prove the liftability nonetheless, because it will assist us in creating the natural extension.

Let $\rho > 1$ be the expansion factor: $d(T(x), T(y)) > \rho d(x,y)$ (where $d$ stands for the Euclidean distance), whenever $x$ and $y$ belong to the same partition element $Z$. Let $S := m^{q-1}(\partial{\mathcal Z})$ be the $(q-1)$-dimensional measure of the hyperplanes forming the partition $\mathcal Z$; this quantity is finite by the assumptions on $\mathcal Z$. For $\eta > 0$ small, let $B(\eta)$ be an $\eta$-neighbourhood of $\partial{\mathcal Z}$ and let $\chi_\eta$ be the indicator function of $B(\eta)$. If $d(x, \partial{\mathcal Z}) < \eta$, and $y \in \partial{\mathcal Z}$ is closest to $x$, then it takes at most $\lceil \frac{\log \eta - \log d_\eta(x,y)}{\log \rho} \rceil$ iterates to move $x$ and $y$ at least $\eta$ apart. Hence, for the first $n$ iterates in the orbit of $x$,
\[ \sum_{j=0}^{n-1} \chi_\eta(T^j(x))  \cdot  \frac{\log \eta - \log d(T^j(x), \partial {\mathcal Z})}{\log \rho},\] 
is an upper bound for the number of iterates $k$ that $T^k x$ is less than $\eta$ away from the image of $\partial{\mathcal Z}$ taken along the same branch $T^{k-j}$ as $T^j x$ at its previous close visit to $\partial{\mathcal Z}$. For the remaining iterates $k$, there is a neighbourhood $U_k \owns x$ such that $T^k$ maps $U_k$ homeomorphically (and in fact affinely) onto an $\eta$-ball around $T^kx$. In other words, $x$ has reached {\em $\eta$-large scale} at time $k$.

By the Ergodic Theorem, for $m^q$-a.e.~$x$,
\[ \begin{array}{l}
\displaystyle \frac1n \sum_{k=0}^{n-1} \chi_\eta(T^k(x))  \cdot \frac{\log \eta - \log d(T^k(x), \partial {\mathcal Z})}{\log \rho}\\
\vspace{.3cm}
\displaystyle \hspace{2cm} \to \frac1{\log \rho}\int_{B(\eta)} (\log \eta - \log d(\xi, \partial {\mathcal Z})) \ d\mu(\xi) \\ 
\vspace{.3cm}
\displaystyle \hspace{2cm} \le  \frac{2S \sup h}{\log \rho}\int_0^\eta (\log \eta - \log \xi) d\xi = \frac{2S \sup h}{\log \rho} \cdot \eta.
\end{array}\]
Thus, the limit frequency that Lebesgue typical points reach $\eta$-large scale is $1-\frac{2 \eta S \sup h}{\log \rho} \ll 1$ for small $\eta$. When lifting the orbit of such typical $x$ to the Hofbauer tower, it will spend a similar proportion of time in a compact part $K$ of the tower, where $K$ depends only on $T$ and $\eta$. In probabilistic terms, the sequence $\left( \frac1n \sum_{k=0}^{n-1} \bar{m^q} \circ \hat T^{-k} \right)_n$ is tight, and this suffices to conclude that Lebesgue measure is liftable, say to $\hat \mu$. Naturally, $\hat \mu \circ \pi^{-1} = \mu$.
\end{proof}

\vskip .3cm
The next two examples show that expansion of Jacobian (rather than uniform expansion in all directions) or having positive Lyapunov exponents can both be insufficient for liftability.

\begin{figure}[ht]
\centering
\subfigure[]{\includegraphics[scale=.65]{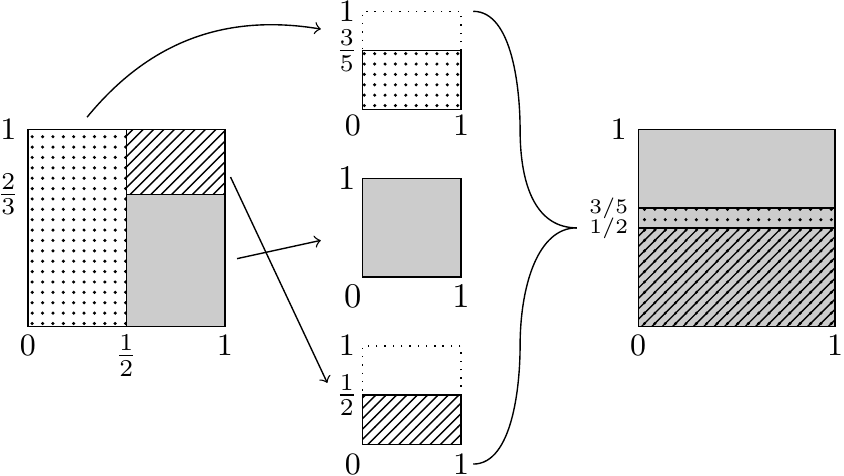}}
\hspace{.8cm}
\subfigure[]{\includegraphics[scale=.6]{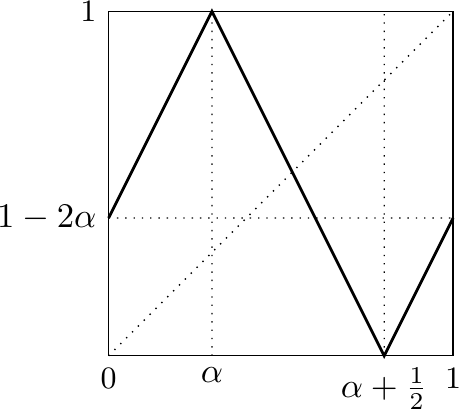}}
\caption{In (a) we see the map from Example~\ref{ex:skew}. (b) shows the map from Example~\ref{ex:raith}, which is a bimodal interval map with slope $\pm2$ and turning points $\alpha$ and $\frac12+\alpha$.}
\label{fig:skewraith}
\end{figure}

\begin{ex}\label{ex:skew}{\rm
Consider the skew product $T:[0,1) \times [0,1)$ defined as
\[ 
T(x,y) = \big(2x (\bmod 1) ,\ a(x)y (\bmod 1)\big)
\quad \text{ for } \quad
a(x) = \begin{cases}
\frac35 & \text{ if } x \in [0, \frac12),\\[2mm]
\frac32 & \text{ if } x \in [\frac12, 1),
\end{cases}
\]
see Figure~\ref{fig:skewraith}(a). This map is transitive and the Jacobian of $T$ w.r.t.\ Lebesgue measure $m^2$ is expanding and locally constant: $J_{m^2,T}(x) = 2a(x)$. Note that for the {\em omega-limit set} $\omega(x,y) = \cap_{n \ge 0} \overline{\{ T^k (x,y)\, : \, k >n\}}$ we have $\omega(x,y) = [0,1) \times \{ 0 \}$ for $m^2$-a.e.~$(x,y)$; this is by a standard argument of skew-products because the Lebesgue typical transversal Lyapunov exponent is $\int \log a(x) \ dx = \log 9/10 < 0$. This implies that the unique weak limit measure of $\frac1n \sum_{k=0}^{n-1} m^2 \circ T^{-k}$ is one-dimensional Lebesgue measure on $[0,1) \times \{ 0 \}$, \ie Lebesgue measure $m^2$ is not liftable.}
\end{ex}

\begin{ex}\label{ex:raith}{\rm
Define the interval map $T:[0,1) \to [0,1)$ by
\[ T(x) = \begin{cases}
2x+1-2\alpha & \text{ if } x \in [0,\alpha),\\
1+\alpha-2x & \text{ if } x \in [\alpha, \alpha + \frac12),\\
2x-2\alpha-1 & \text{ if } x \in [\alpha + \frac12, 1),\\
\end{cases}\]
see Figure~\ref{fig:skewraith}(b). For certain values of $\alpha$, the set of points whose orbits stay in $[0,\alpha] \cup [\alpha + \frac12, 1)$ form a Cantor set $H$ of zero Hausdorff 
dimension on which $T$ is semi-conjugate to a circle rotation. The measure obtained from lifting Lebesgue measure to $H$ is invariant (hence of Jacobian $1$), has zero entropy but Lyapunov exponent $\log 2$. It is not liftable to the Hofbauer tower. This example was inspired by \cite{HR89}, see also \cite{BT09}.
}\end{ex}

\iffalse
\begin{remark}{\rm Condition (c4) is needlessly restrictive, because we will only need that the Jacobian is locally constant for the dynamics on the Hofbauer tower $(\hat X, \hat T)$, so that its invariant measure is effectively a Markov measure. For this reason, we could, for a fixed $n_0 \ge 2$, relax (c4) to $n$-cylinders,
\ie
\begin{itemize}
\item[(c4')] There are constants $s_{j_0\dots j_{n_0-1}} \geq 1$ such that $J_{\mu,T}|_{Z_{j_0\dots j_{n_0-1}}} := \frac{d\mu \circ T}{d\mu}|_{Z_{j_0\dots j_{n_0-1}}} = s_{j_0\dots j_{n_0-1}}$. 
\end{itemize}
The reader can easily supply the necessary adjustments for the next lemma and proposition if (c4) is relaxed to (c4'), but to avoid cumbersome argument, and since (c4) is sufficient for all examples in this paper, we will use (c4).}
\end{remark}
\fi

\subsection{Piecewise constant Radon-Nikodym derivatives}\label{ss:radonnikodym}
Assumption (c4) implies that if $D \to D'$ and $D \subset Z_j$, then the Jacobian $J_{\bar \mu, \hat T} = s_j$ on $D$. The next lemma shows that the Radon-Nikodym derivative $\frac{d\hat \mu}{d\bar \mu}$ is constant on $D$ as well. We need this for the construction of the natural extension.

\begin{lemma}\label{l:mun}
If (c1)-(c4) hold for the system $(X, \mathcal B, \mu, T)$, then $\hat \mu_n \ll \bar \mu$ for each $n \ge 1$. Moreover, the densities $\frac{d\hat \mu_n}{d\bar \mu}$ are constant on each $D \in \mathcal D$.
\end{lemma}

\begin{proof}
Fix $D \in \mathcal D$ and take a measurable set $A \subseteq D$. Note that for each $(k+1)$-cylinder $Z_{j_0 \cdots j_k}\in {\mathcal Z}_k$ with $T^kZ_{j_0 \cdots j_k}=D$, by (c4) we have
\[ (s_{j_0}\cdots s_{j_{k-1}}) \, \mu\big( \pi(\hat T^{-k}A \cap \hat X_0)\cap Z_{j_0\cdots j_k} \big) 
= \mu\big(\pi(A)\big) = \bar \mu(A), \]
where an empty product $s_{j_0}\cdots s_{j_{k-1}}$ for $k=0$ is taken as $1$. Hence,
\begin{eqnarray*}
\bar \mu \big( \hat T^{-k}A \cap \hat X_0 \big) &=& 
\sum_{\stackrel{ Z_{j_0\cdots j_k} \in \mathcal Z_k:}{T^k Z_{j_0 \cdots j_k}\ = \, D}}  \mu\big(\pi(\hat T^{-k}A\cap \hat X_0)\cap Z_{j_0\cdots j_k}\big)\\
&=& \sum_{\stackrel{ Z_{j_0\cdots j_k} \in \mathcal Z_k:}{T^k Z_{j_0 \cdots j_k}\ = \, D}} \frac{\bar \mu(A)}{s_{j_0} \cdots \, s_{j_{k-1}}}.
\end{eqnarray*}
Passing to the Cesaro mean, this implies that $\hat \mu_n \ll \bar \mu$. Also, we can write $\hat \mu_n(A) = \int_A \, \rho_n(D)\, d\bar \mu$ with
\begin{equation}\label{q:pwdensity}
\rho_n(D):=  \frac{1}{n} \, \sum_{k=0}^{n-1} \, \sum_{\stackrel{ Z_{j_0 \cdots j_k}:}{T^k Z_{j_0 \cdots j_k}\ = \, D}} 
(s_{j_0} \cdots \, s_{j_{k-1}})^{-1}.
\end{equation}
Since $\rho_n(D)$ only depends on $D$, we get the lemma with $\frac{d \hat \mu_n}{d \bar \mu} \big|_D= \rho_n(D)$. 
\end{proof}

\begin{proposition}\label{p:rhoD}
Assume that (c1)-(c4) hold for $(X, \mathcal B, \mu, T)$, and that $\hat \mu$ is a non-zero vague limit point of $\{ \hat \mu_n\}_{n \ge 1}$. Then $\hat \mu \ll \bar \mu$ and the density $\frac{d\hat \mu}{d \bar \mu}$ is constant on each set $D \in \mathcal D$ and given by $\rho(D):=\frac{d\hat \mu}{d \bar \mu}\big|_D = \frac{\hat \mu(D)}{\bar \mu(D)}$. 
\end{proposition}

\begin{remark}{\rm Note that $\hat \mu \ll \bar \mu$ implies that $\hat \mu \circ \pi^{-1} \ll \mu$. The previous proposition doesn't use $T$-invariance of $\mu$. Since $\hat \mu$ is $\hat T$-invariant even if $\mu$ is not $T$-invariant, $ \hat \mu \circ  \pi^{-1} \ll \mu$ is $T$-invariant and in the sequel we produce a natural extension of $(X,{\mathcal B}, \hat \mu \circ  \pi^{-1} , T)$.}
\end{remark}

\begin{proof}
Fix $D \in \mathcal D$ and $A \subset D$ compact. By Lemma~\ref{l:mun}, $\rho_n(D)$ is constant, and since $\hat \mu$ is a vague limit point of the sequence $\{ \hat\mu_n\}_{n \ge 1}$ along some subsequence $\{n_k\}_{k \ge 1}$, $\{ \hat\mu_{n_k}|_A\}_{k \ge 1}$ converges to $\hat \mu|_A$ in the 
weak topology as $k \to \infty$. This means that $\rho_{n_k}(D) = \frac{\hat \mu_{n_k}(A)}{\bar \mu(A)}$ converges to a constant limit density $\rho(D)$.
Clearly $\hat\mu(D) = \rho(D) \bar\mu(D)$, so $\rho(D):=\frac{d\hat \mu}{d \bar \mu}\big|_D = \frac{\hat \mu(D)}{\bar \mu(D)}$ follows. 
\end{proof}

\subsection{The natural extension}\label{ss:natural_extension}
From the Hofbauer tower we will obtain a version of the {\em natural extension} of the transformation $T$. We start from a system $(X, \mathcal B, \mu , T)$ satisfying (c1)-(c6) and {\bf we assume that the measure $\mu$ is liftable}, either by satisfying the requirements of Theorem~\ref{thm:Keller3} if $\mu$ is $T$-invariant or by other means (such as Proposition~\ref{prop:noninvariant}). Let us first give a formal definition of the natural extension. 

\begin{definition}\label{def:naturalextension}
A measure theoretical dynamical system $(Y, \mathcal C, \nu, F)$ is a {\em natural extension} of the non-invertible system $(X, \mathcal B, \mu, T)$ if all the following are satisfied. There are sets $X^* \in \mathcal B$ and $Y^* \in \mathcal C$, with 
$\mu(X^*)=1=\nu(Y^*)$ and $T (X^*) \subset X^*$ and $F (Y^*) \subseteq Y^*$ and there is a map 
$\phi: X^* \to Y^*$ such that
\begin{itemize}
\item[(ne1)] $F$ is invertible $\nu$-a.e.;
\item[(ne2)] $\phi$ is bi-measurable and surjective;
\item[(ne3)] $\phi$ preserves the measure structure, \ie $\mu = \nu \circ \phi^{-1}$;
\item[(ne4)] $\phi$ preserves the dynamics, \ie $\phi \circ T = F \circ \phi$;
\item[(ne5)]  $\mathcal C$ is the coarsest $\sigma$-algebra that makes
(ne1)-(ne4) valid, \ie \\
 $\bigvee_{n \geq 0} F^{-n}(\phi^{-1}{\mathcal B}) = \mathcal C$.
\end{itemize}
If a map $\phi$ satisfies (ne2), (ne3), (ne4) and is injective, then the systems $(Y, \mathcal C, \nu, F)$ and $(X, \mathcal B, \mu, T)$ are called {\em isomorphic} and $\phi$ is an {\em isomorphism}.
\end{definition}
\vskip .2cm
For the natural extension only the components $D_u$ of positive $\hat \mu$-measure are important, but w.l.o.g.~we can assume that $\hat \mu(D_u) > 0$ for all $u \in \alpha$. To define the natural extension domain $Y$, extend each $D_u$ by one dimension to a set $R_u = D_u \times [0, \rho(D_u)]$, which we will call the {\em rug} of $D_u$. Set $Y := \bigsqcup_{u \in \alpha} R_u$ and let $\mathcal C$ denote the Borel $\sigma$-algebra on $Y$. Use $\nu$ to denote the product measure on $(Y, \mathcal C)$ given by $\bar \mu \times m$ on each rug, where $m$ is the one-dimensional Lebesgue measure. Then by Proposition~\ref{p:rhoD}
\[ \nu (Y) =\sum_{u \in \alpha} \bar \mu(D_u)\rho(D_u) = \sum_{u \in \alpha} \hat \mu (D_u)=1. \]
Let $\hat \pi:Y \to \hat X$ be the projection onto the first coordinate. Then $\nu \circ \hat \pi^{-1}=\hat \mu$. We will extend the action of $\hat T$ to the vertical direction, obtaining a new map which we call $F: Y \to Y$. This is done piecewise as follows. For $z = (\hat x, y, u) \in R_u$ with $\pi(\hat x) \in Z_j$ and $\hat T \hat x \in D_v$, define
\begin{equation}\label{q:natexmap}
F z= F(\hat x, y, u) = \left( \hat T \hat x\ ,\ \frac{y}{s_j} + 
\sum_{1 \le k \le N} \sum_{\stackrel{ t<u :}{\pi^{-1}(Z_k) \supset D_t \to D_v}} \frac{ \rho(D_t)}{s_k}\ ,\ v \right).
\end{equation}
In words, the parts of all the rugs $R_t$ that map to $R_v$ are squeezed in the vertical direction by a factor equal to the expansion in the `horizontal' direction and are stacked on top of each other into the rug $R_v$ according to the order relation on $\mathcal D$, see Figure~\ref{f:natexex}. Hence, the image strips in $R_v$ are disjoint. By Proposition~\ref{p:rhoD} the map $F$ is well-defined on a full measure subset of $Y$. Since the stretch in the horizontal direction and the squeeze in vertical direction are the same, $F$ preserves area $\nu$. The next lemma gives (ne1).
\begin{figure}[ht]
\centering
\includegraphics{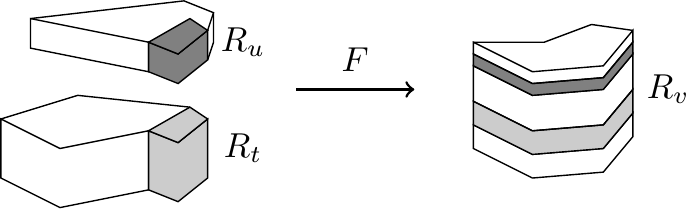}
\caption{$F$ maps the coloured regions in $R_t$ and $R_u$ both to $R_v$, \ie there is a set $Z_j$, such that $TD_t \cap Z_j = TD_u \cap Z_j=D_v$ . If $t < u$, then the image of $R_t$ in $R_v$ lies below the image of $R_u$.  Also, the image of $R_t$ and $R_u$ stretch all the way across $R_v$ in the ``horizontal'' direction.}
\label{f:natexex}
\end{figure}

\begin{lemma}\label{l:invertible}
The map $F$ is invertible  $\nu$-a.e.
\end{lemma}

\begin{proof}
To show that $F$ is surjective, first note that
\begin{eqnarray}\label{q:hatmudu}
\hat \mu (D_u) &=& \hat \mu (\hat T^{-1}D_u) = \sum_{t \in \alpha} \int_{\hat X} \frac{d\hat \mu}{d\bar \mu}\, 1_{\hat T^{-1}D_u \cap D_t} d\bar \mu \nonumber \\
 &=& \sum_{t \in \alpha} \rho(D_t) \, \bar \mu \big( D_t \cap \hat T^{-1} D_u \big) .
\end{eqnarray}
Thus,
\begin{eqnarray*}
\rho(D_u) &=& \frac{1}{\bar \mu(D_u)} \sum_{t \in \alpha} \rho(D_t) \, \bar \mu( D_t \cap \hat T^{-1} D_u )\\
&=& \frac{1}{\bar \mu(D_u)} \sum_{1 \le j \le N} \sum_{\stackrel{ t\in \alpha :}{\pi^{-1}(Z_j) \supset D_t \to D_u}} \rho(D_t) \frac{\bar \mu(D_u)}{s_j} = \sum_{t \in \alpha}  \sum_{ \stackrel{ 1 \le j  \le N :}{\pi^{-1}Z_j \supset D_t \to D_u} } \frac{\rho(D_t)}{s_j}.
\end{eqnarray*}
This shows that every $(\hat x,y) \in R_u$ is the images of something; the $\hat x$-coordinate because $T D = D_u$ if $D \to D_u$, and the $y$-coordinate because only those $D_t$ with $D_t \to D_u$ contribute to the strips of the rug $R_u$.

For the injectivity of $F$, first note that for the horizontal boundary of each rug $R_u$ we have $\nu \big( D_u \times \{0,\rho(D_u)\} \big) = \bar \mu(D_u) \cdot 0=0$. Let $M$ be the union of all these boundaries, \ie $M = \cup_{u \in \alpha} \big( D_u \times \{ 0,\rho(D_u)\} \big)$. Then $\nu \big( \cup_{n \in \Z} F^n M \big)=0$. Assume that $(\hat x_1, y_1, t), ( \hat x_2, y_2,u) \in Y \setminus M$ 
are such that
\[ F(\hat x_1, y_1, t)= F( \hat x_2, y_2,u) = (\hat x, y , v).\]
Then, by the injectivity of $T$ on each of the elements of $\mathcal Z = \{ Z_j \}_{j=1}^N$, if $ \hat x_1 \neq \hat x_2$ with $ \hat x_1 \in D_t \subset \pi^{-1} (Z_j)$ and $ \hat x_2 \in D_u \subset \pi^{-1}(Z_k)$, then either $t \neq u$ or $j \neq k$. By the definition of $F$, either one of these inequalities implies that the second coordinates of $F(\hat x_1, y_1, t)$ and $F(\hat x_2, y_2,u)$ cannot be equal. Hence $t=u$ and $\hat x_1 = \hat x_2$. Since $F$ stacks the rugs on top of each other according to the ordering on $\mathcal D$, this implies that also $y_1=y_2$. Hence, $(\hat x_1, y_1, t) = (\hat x_2, y_2,u) $ and $F$ is invertible.
\end{proof}

\vskip .3cm
Note that $\hat T \circ \hat \pi = \hat \pi \circ F$, where as before, $\hat \pi:Y \to \hat X$ is the projection onto the first coordinate. Also, $\nu(\hat \pi^{-1}(B)) = \hat \mu(B)$ for all $B \in \hat{\mathcal B}$. Let $\phi:= \pi \circ \hat \pi: Y \to X$, see Figure~\ref{Fig:diagram}. Since also $T \circ \pi = \pi \circ \hat T$ and $\hat \mu \circ \pi^{-1}=\mu$, $\phi$ satisfies (ne2), (ne3) and (ne4).
\begin{figure}[ht]
\centering
\includegraphics{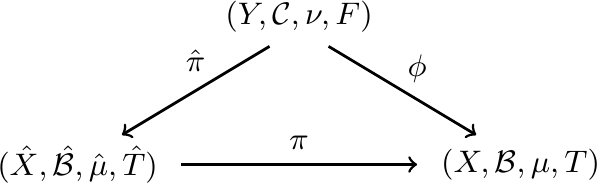}
\caption{The three spaces involved and the projections between them.}
\label{Fig:diagram}
\end{figure}

It remains to show (ne5). Recall the definition of the sets $\partial_n D_u$ from (\ref{q:partialdn}) and the sets $B_{u,n}$ from (\ref{q:bun}). Let $B_n = \bigsqcup_{u \in \alpha}B_{u,n}$ be the disjoint union of these sets over all $u$. Then $B_{n+1} \subseteq B_n$ and by (c6) and Proposition~\ref{p:rhoD}, $\lim_{n \to \infty} \hat \mu(B_n) =0$. We extend the sets $B_{u,n}$ to $Y$ by defining
\[ E_{u,n} = \bigcup_{Z \in \mathcal Z_n} \big\{ (Z \cap D_u)\times [0, \rho(D_u)] \, : \, 0< \bar \mu (Z \cap D_u) < \mu (Z) \big\} \subseteq Y.\]
Also, let $E_n = \bigsqcup_{u \in \alpha}E_{u,n}$. Since $\hat \mu=\nu \circ \hat \pi^{-1}$, we have
\[ 
0 \le \nu(\bigcap_{n \ge 0} E_n) = \lim_{n \to \infty} \nu(E_n) \le \lim_{n \to \infty} (\nu \circ \hat \pi^{-1})(B_n) = \lim_{n \to \infty} \hat \mu(B_n)=0.\]
We need the following lemma.

\begin{lemma}\label{l:bn}
If $z \in E_{n+1}$, then $F z \in E_n$ and hence $F^n E_{n+1} \subseteq F^{n-1}E_n$.
\end{lemma}

\begin{proof}
Let $z \in E_{n+1}$ and suppose $\hat \pi(z) \in Z_{j_0 \cdots j_n} \cap D_u$ for some set $Z_{j_0 \cdots j_n} \in \mathcal Z_{n+1}$ and $u \in \alpha$. Then there is a $v \in\alpha$, such that $\hat T D_u = D_v\subset \pi^{-1}(Z_{j_1})$ and $\hat \pi(Fz)\in \hat T \big(\pi^{-1}(Z_{j_0 \cdots j_n})\cap D_u \big) = \pi^{-1}(Z_{j_1 \cdots j_n})\cap D_v$. Then,
\begin{eqnarray*}
0 < \bar \mu \big(\pi^{-1}(Z_{j_1 \cdots j_n})\cap D_v \big) &=& s_{j_0} \bar \mu \big(\pi^{-1}(Z_{j_0 \cdots j_n})\cap D_u\big) \\
&<&  s_{j_0} \mu( Z_{j_0 \cdots j_n}) \le \mu(Z_{j_1 \cdots j_n}),
\end{eqnarray*}
which gives that $Fz \in E_n$. 
\end{proof}

\vskip .3cm
For points $z \in E_n$, the first $n$ iterates $F^k z$ lie close to the boundary of their rugs. We will prove (ne5) by showing that the map $F$ separates points. In order to make this work, we need to exclude points of which all inverse images lie close to the boundary.

\begin{lemma}\label{l:sigmaalgebra}
Under condition (c6) the $\sigma$-algebra $\bigvee_{n \geq 0} F^n (\phi^{-1}{\mathcal B})$ is equal, up to a set of $\nu$-measure zero, to the $\sigma$-algebra $\mathcal C$ of Lebesgue measurable sets on $Y$.
\end{lemma}

\begin{proof}
First we define the exceptional set. Let $ E = \bigcap_{n \ge 1} F^{n-1}E_n$. Since $F$ is invertible almost everywhere, $\nu(F^{n-1}E_n)=\nu(E_n)$ for all $n \ge 1$. Lemma~\ref{l:bn} implies that $F^n E_{n+1} \subseteq F^{n-1}E_n$ for all $n \ge 1$. Therefore
\[ \nu( E) = \lim_{n \to \infty} \nu(F^{n-1} E_n) = \lim_{n \to \infty} \nu(E_n)=0.\]
Let $z = (\hat x, y, u)\in R_u \setminus (F E)$ and $z' = (\hat x', y', u') \in R_{u'} \setminus (F E)$ be two different points in $Y$. It suffices to show that there are sets $B$ and $B' \in \mathcal B$ and $n\ge 1$, such that $z \in F^n \phi^{-1}(B)$ and $z' \in F^n \phi^{-1}(B')$ and moreover $F^n \phi^{-1}(B) \cap F^n \phi^{-1}(B') = \emptyset$.

\vskip .2cm

Note that if $x = \phi (z) \neq x' =\phi (z')$, then there are two disjoint open sets $B,B' \subset X$, such that $x \in B$ and $x' \in B'$. Then $\phi^{-1}(B)$ and $\phi^{-1}(B')$ are still disjoint and contain $z$ and $z'$ respectively.

Now suppose that $\phi(z) = \phi(z')$, but $u \neq u'$. We introduce some notation. For $n \ge 1$, write $F^{-n} z \in R_{u_n}$ and use $Z^{(u_n)}$ to denote the $(n+1)$-cylinder in which $\phi(F^{-n}\hat z)$ lies. For $z'$, we use $R_{u_n'}$ and $Z^{(u_n')}$ respectively. Suppose that $\phi (F^{-n}z)=\phi(F^{-n}z')$ for all $n \ge 1$, \ie $Z^{(u_n)}=Z^{(u'_n)}$ for all $n \ge 1$. Since $z, z' \not \in E$, there are $k,k' \ge 1$, such that $z \not \in F^{k-1} E_k$ and $z' \not \in F^{k'-1} E_{k'}$. By Lemma~\ref{l:bn} we have for all $n \ge \max\{k,k'\}$, that $F^{-n+1}z,F^{-n+1}z' \not \in E_n$. This implies that
\[\bar \mu \big(\pi^{-1}(Z^{(u_n)})\cap D_{u_n}\big) = \mu(Z^{(u_n)}) \quad \text{and} \quad \bar \mu \big(\pi^{-1}(Z^{(u_n')})\cap D_{u_n'} \big) = \mu(Z^{(u_n')})\]
and that $Z^{(u_n)}=Z^{(u_n')}$. Thus,
\[ D_u = T^n (Z^{(u_n)}\cap D_{u_n})=T^n (Z^{(u_n')}\cap D_{u_n'})=D_{u'},\]
a contradiction. Hence, there is an $n$ such that $\phi (F^{-n}z)\neq \phi(F^{-n}z')$. This means that we can find two disjoint open sets $B,B' \subset X$, such that $\phi (F^{-n}z) \in B$ and $\phi (F^{-n}z') \in B'$. Then again $F^n\phi^{-1}(B)$ and $F^n\phi^{-1}(B')$ are still disjoint and contain $z$ and $z'$ respectively.

Finally, if $z$ and $z'$ are in the same rug with $ y \neq y'$, then, since $F$ is contracting in the vertical direction, there is an $n$, such that $F^{-n} z$ and $F^{-n}z'$ are in different rugs and we can repeat the argument from above.
\end{proof}

\vskip .3cm
This lemma finishes the proof of the following theorem.
\begin{theorem}
Let $(X, \mathcal B, \mu, T)$ be a system that satisfies conditions (c1)-(c6) and assume $\mu$ lifts to a probability measure $\hat \mu$ on $(\hat X, \hat{\mathcal B})$. Then the system $(Y, \mathcal C, \nu, F)$ is the natural extension of $(X, \mathcal B, \hat \mu \circ \pi^{-1}, T)$ with factor map $\phi = \pi \circ \hat \pi$. In case $\mu$ is $T$-invariant, then $\hat \mu$ is the unique lift and $\hat \mu \circ \pi^{-1}=\mu$.
\end{theorem}

\begin{corollary}
Let $(X, \mathcal B, \mu, T)$ be a system that satisfies conditions (c1)-(c6). If $\mu$ is $T$-invariant and $h_{\mu}(T) > \text{cap}(T)$, then $\nu$ is ergodic and $h_{\nu}(F)=h_{\mu}(T)$.
\end{corollary}

\begin{proof}
If $\mu$ is $T$-invariant, then $\hat \mu \circ \pi^{-1}=\mu$. By (c5) $\mu$ is ergodic and since ergodicity and metric entropy are preserved under taking the natural extension (see \cite{Roh64}), the result follows. 
\end{proof}

\begin{remark}{\rm
The fact that the natural extension of $(X,T)$ in general contains more points, \ie more backward orbits, than the natural extension of $(\hat X,\hat T)$ was already observed by Buzzi \cite{Buz97,Buz99}. He shows that the set of points in the natural extension of $(X,T)$ that are not represented in the natural extension of the Markov extension carry no measure of positive entropy (or of entropy near the maximal entropy for fairly general higher dimensional systems).
}\end{remark}

\section{Bernoulli-like properties} \label{sec:Bernoulli}
In this section we will discuss Bernoulli-like properties of the natural extension and how to transfer them from the natural extension to the original system and back. Let us first recall some definitions.

By a {\em two-sided (resp.\ one-sided) Bernoulli shift} we mean a shift space $({\mathcal A}^\Z, \sigma)$ (resp.\ $({\mathcal A}^{\N_0}, \sigma)$) on a finite or countable alphabet $\mathcal A$ with left shift $\sigma$, and equipped with a stationary product measure based on a probability vector $(p_1, \dots, p_n)$. An invertible dynamical system $(Y, \mathcal C, \nu, F)$ that is isomorphic to a Bernoulli shift is called {\em Bernoulli} itself.

If  $(X, {\mathcal B}, \mu, T)$ is non-invertible, and isomorphic to a one-sided Bernoulli shift, then it is called {\em one-sided Bernoulli} itself. This is a much stronger property than the natural extension of $(X, {\mathcal B}, \mu, T)$ being isomorphic to a two-sided Bernoulli shift (cf.\ \cite{BH09}); if the latter happens, the non-invertible system is called Bernoulli. It is a well-known result by Ornstein \cite{Orn70} 
(and \cite{Smo72} and \cite{Orn71} for infinite alphabets) that entropy is a complete invariant for two-sided Bernoulli systems with positive or infinite entropy, but this is not true in general for the non-invertible case.

One theorem for which having a geometric version of the natural extension is useful is Theorem 3 from \cite{Sal73} by Saleski. To apply this theorem, we first show that the natural extension $F$ has an induced transformation that is Bernoulli. Consider one of the rugs $R_u$, $u \in \alpha$, and define the first return times for $z \in R_u$ under $F$ as
\[ 
\tau_u(z) = \inf \{ n \ge 1\, : \, F^n z \in R_u\}.
\]
By the Poincar\'e Recurrence Theorem, $\tau_u(z)< \infty$ for $\nu$-a.e.~$z \in R_u$. %; in fact, Kac's Lemma says that $\int_{R_u} \tau_u d\hat \mu = 1$.  
Define the induced map $F_u:R_u \to R_u$ by $F_u z = F^{\tau_u(z)}z$.
\begin{theorem}\label{t:induced}
Each map $F_u: R_u \to R_u$, $u \in \alpha$, is Bernoulli.
\end{theorem}

\begin{proof}
Consider the partition $\mathcal P = \{ P_1, P_2, \ldots \}$ of $D_u$ into sets $P_n$ such that
\[ P_n = \{ \hat \pi(z) \in D_u \, : \, \tau_u(z)=n\}.\]
The map $\phi$ from Definition~\ref{def:naturalextension}
acts as projection $\phi:R_u \to D_u$.
For each $z \in R_u$ with $\tau_u(z)=n$, there is a corresponding $n$-path in $(\mathcal D, \to)$ from $D_u$ to $D_u$. Hence, there is an ($n+1$)-cylinder $Z_{j_0 \cdots j_n}$, such that $\phi(z) \in Z_{j_0 \cdots j_n}$ and $\hat T^n\big(\pi^{-1}(Z_{j_0 \cdots j_n})\cap D_u\big)=D_u$. Therefore, each set $P_n$ can be written as a finite union of pairwise disjoint sets:
\[ 
P_n = \bigcup \{ \pi^{-1}(Z_{j_0 \cdots j_n})\cap D_u \, : \, Z_{j_0 \cdots j_n} \in \mathcal Z_{n+1}, \, 
T^n Z_{j_0 \cdots j_n} =\pi(D_u) \}.
\]
Note that the value of $\tau_u(z)$ for $z=(\hat x,y,u) \in R_u$ does not depend on $y$. Thus, we can write $\tau_u(z) = \tau_u(\hat x)$ for almost all $\hat x \in D_u$. Define the map $\hat T_u \hat x = \hat T^{\tau_u(\hat x)} \hat x$. Then $\hat T_u \big( \pi^{-1}(Z_{j_0 \cdots j_n})\cap D_u\big)=D_u$ and thus $\hat T_u$ is Bernoulli.

Using the same arguments as before, we see that $(R_u, \mathcal C \cap R_u, \nu_u:=\nu|_{R_u}, F_u)$ is the natural extension of $(D_u, \hat{\mathcal B} \cap D_u, \hat \mu|_{D_u}, \hat T_u)$ with factor map $\hat \pi$. Since $\hat T_u$ is Bernoulli, $F_u$ is Bernoulli as well.
\end{proof}

\vskip .3cm
\noindent Theorem~\ref{t:induced} combined with Saleski's result implies the following.

\begin{theorem}[Saleski \cite{Sal73}]
Suppose $(Y, {\mathcal C}, \nu, F)$ is weakly mixing. Fix $u \in \alpha$ and suppose that the following entropy condition holds:
\[ H_{\nu_u} \big( \vee_{k=1}^{\infty} \vee_{n=1}^{\infty} F^k_u \, Y_n | \vee_{i=1}^{\infty} F^i_u \mathcal P \big) < \infty,\]
where $Y_n = \{ R_u-\cup_{j=1}^n F^{-j} R_u, R_u \cap \cup_{j=1}^n F^{-j}R_u \}$ and $\mathcal P$ is a Bernoulli partition of $(R_u, F_u)$. Then $F$ is a Bernoulli automorphism and hence $T$ is Bernoulli as well.
\end{theorem}

Recall the construction of the Markov shift at the end of Section~\ref{ss:Hofbauer}. The invertibility of $F$ allows us to associate to $F$ a two-sided countable state topological Markov shift and to use all the results available for this type of maps. To construct this Markov shift, first assign to a.e.~$z \in Y$ a two-sided sequence $b(z)=(b_k)_{k \in \Z}$ by setting $b_k = u$ if $F^k z \in R_u$. Define the map $\psi : Y \to \alpha^\Z$ by $\psi(z) = b(z)$ and let $\Omega = \psi(Y)$. On $\Omega$, let $\mathcal P$ denote the product $\sigma$-algebra and let $\sigma$ be the left shift as usual. The Markov measure $m_{\underline v, P}$ is given by the probability vector $\underline v = (v_u)_{u \in \alpha}$ with entries $v_u = \hat \mu(D_u)$ and the (possibly) infinite probability matrix $P = (p_{t,u})_{t,u \in \alpha}$ defined by
\[ p_{t,u} = \frac{\bar \mu( D_t \cap \hat T^{-1}D_u) }{\bar \mu(D_t)} 
%= \frac{1}{\bar \mu(D_t)}\sum_{\stackrel{1 \le j \le N:}{\overline{\hat T (D_t \cap \pi^{-1}(Z_j))}=D_u}} s_j^{-1} \cdot \bar \mu(D_u)
,\]
\begin{proposition}
The systems $(Y, \mathcal C, \nu, F)$ and $(\Omega, \mathcal P, m_{\underline v}, \sigma)$ are isomorphic with isomorphism $\psi$, \ie $\psi$ satisfies (ne2), (ne3) and (ne4) and is injective.
\end{proposition}

\begin{proof}
To show that $\psi$ is $\nu$-a.e.~injective, note that $\psi(z)=\psi(z')$ implies that $F^n z, F^n z' \in D_{u_n}\subseteq Z_{j_n}$ for some sequence $(u_n) \in \alpha^{\Z}$. Since $F$ is expanding in the horizontal direction, this is only possible if $z = z'$. It is immediate that $\psi$ is surjective and bi-measurable and that $\psi \circ F = \sigma \circ \psi$. Furthermore, it is straightforward to check that $\nu \circ \psi^{-1} = m_{\underline v}$. Hence, $\psi$ is a bi-measurable bijection that satisfies (ne2), (ne3) and (ne4) and is thus an isomorphism. 
\end{proof}

\vskip .3cm
A (non-invertible) map $T$ is called {\em exact} on $(X, \mathcal B, \mu)$ if $\cap_{n=1 }^{\infty} T^{-n} \mathcal B = \{ \emptyset, X\}$. An invertible map $F$ is called a {\em $K$-automorphism} on $(Y, {\mathcal C}, \nu)$ if there is a sub-$\sigma$-algebra $\mathcal C_0 \subseteq \mathcal C$ satisfying (i) $F^{-1} \mathcal C_0 \subseteq \mathcal C_0$, (ii) $\cap_{n=1 }^{\infty} F^{-n} \mathcal C_0 = \{ \emptyset, Y\}$ and (iii) the $\sigma$-algebra generated by $\cup_{n=1}^{\infty} F^n \mathcal C_0$ equals $\mathcal C$. An ergodic Markov shift is a $K$-automorphism if and only if it is strong mixing (see \cite{Ito87} for example). A result from Rohlin \cite{Roh64} says that a map is a $K$-automorphism if and only if it is the natural extension of an exact transformation. This gives the following corollary.

\begin{corollary}
If $(X, \mathcal B, \mu, T)$ satisfies (c1)-(c6) and $\mu$ is $T$-invariant and liftable, then $T$ is exact if and only if $F$ is a $K$-automorphism if and only if $F$ is strongly mixing if and only if the associated Markov shift $(\Omega, \mathcal P, m_{\underline v}, \sigma)$ is irreducible and aperiodic. 
\end{corollary}

\section{Interval maps}\label{sec:interval}

In this section we apply the construction of the natural extension to the specific case of piecewise linear expanding interval maps.

\subsection{Piecewise linear interval maps}\label{ss:intervalmaps}
Let $T:[0,1] \to [0,1]$ be a piecewise linear expanding map. Let the partition $\mathcal Z$ consist of the closures of all maximal intervals of monotonicity for $T$. Then each set $D \in \mathcal D$ is an interval. Conditions (c1)-(c5) are immediate for Lebesgue measure $m$. Each set $Z_j \in \mathcal Z$ is an interval and thus $T^n( \cup_{j=1}^N \partial Z_j )$ consists of only finitely many points and $m \big( T^n ( \cup_{j=1}^N \partial Z_j ) \big)=0$. This implies (c6). Since $T$ is a piecewise linear expanding interval map Proposition~\ref{prop:noninvariant} applies and $m$ is liftable. Hence the construction of the natural extension from Section \ref{ss:natural_extension} applies.

There always exists an ergodic invariant measure $\mu \ll m$ on $[0,1]$. This measure satisfies all conditions except possibly (c4). Recall the definition of capacity from \eqref{q:cap}. Note that $\#(\partial_n D)=2$ for all $n$ and hence $\text{cap}(T) =0$. Then a $T$-invariant measure $\mu$ is liftable whenever $h_{\mu}(T)>0$. This was first proved by Hofbauer in \cite{Hof79}. Hence, if one can show (c4) in a particular case, then one could also use the measure $\mu$.

This class of maps includes any piecewise linear expanding map of which the absolute value of the slope is constant. Here the entropy is equal to the log of the absolute value of the slope. For such maps, there is an additional result by Rychlik \cite{Ryc83}: If the natural extension map $F$ is a $K$-automorphism, then $F$ is {\em weakly Bernoulli}, \ie for each $\varepsilon >0$ there is a positive integer $N$, such that for all $m \ge 1$, all sets $A \in \bigvee_{k=0}^m F^{-k} \mathcal C$ and $C \in \bigvee_{k=-N-m}^{-N} F^{-k} \mathcal C$ we have
\[ |\nu(A \cap C)- \nu(A)\nu(C)|< \varepsilon,\]
where $\bigvee_{i=\ell}^k F^{-i}\mathcal C$ denotes the smallest $\sigma$-algebra containing all elements in $F^{-i}\mathcal C$. The weak Bernoullicity of $F$ implies that of $T$.

\subsection{Positive and negative slope $\beta$-transformations}\label{ss:beta}
Let $1<\beta<2$. The positive slope $\beta$-transformation is defined by $T_{\beta}x = \beta x \pmod 1$. It is a very well studied map with many interesting properties. It has a unique measure of maximal entropy $\mu_1$, equivalent to $m$, with entropy $h_{\mu_1}(T)=\log \beta$. Note that the Hofbauer tower gives a $T$-invariant measure $\hat \mu \circ \pi^{-1}$ by lifting $m$. By ergodic decomposition this measure is either equal to $\mu_1$ or not ergodic.
\begin{figure}[ht]
\centering
\includegraphics{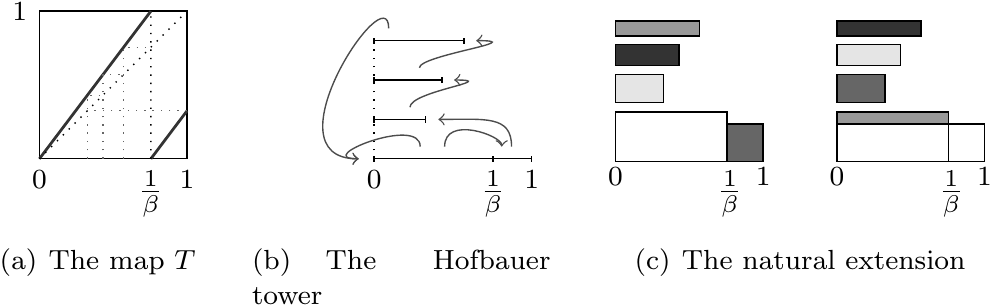}
\caption{In (a) we see the positive $\beta$-transformation with $\beta$ equal to the real root of $x^3-x-1$. (c) shows its natural extension. The transformation $F$ maps the areas on the left to the areas on the right with the same colour.}
\label{f:posbeta}
\end{figure}
In Figure~\ref{f:posbeta} we see an example of a positive slope $\beta$-transformation for a specific value of $\beta$ and its natural extension. In \cite{DKS} Dajani et al.~used a similar version of the geometric the natural extension of $T$ and showed that it is Bernoulli. Since all natural extensions of the same system are isomorphic (\cite{Roh64}), the natural extension given here is also Bernoulli.

\begin{remark}{\rm Note that we start with Lebesgue measure $m$ on the unit interval. By Proposition~\ref{p:erg} the construction produces an ergodic invariant probability measure $\hat \mu \circ \pi^{-1}$ for $T_{\beta}$. Since $\hat \mu \circ \pi^{-1} \ll m$, and there is only one measure with these properties, we automatically have that $\hat \mu \circ \pi^{-1}=\mu_1$.
}\end{remark}

\vskip .2cm
The negative $\beta$-transformation is defined on the unit interval $[0,1]$ by $Sx = -\beta x \pmod 1$. It has a unique measure of maximal entropy $\mu_2$, absolutely continuous with respect to Lebesgue. Also for this map $h_{\mu_2}(S)=\log \beta$.
\begin{figure}[ht]
\centering
\includegraphics{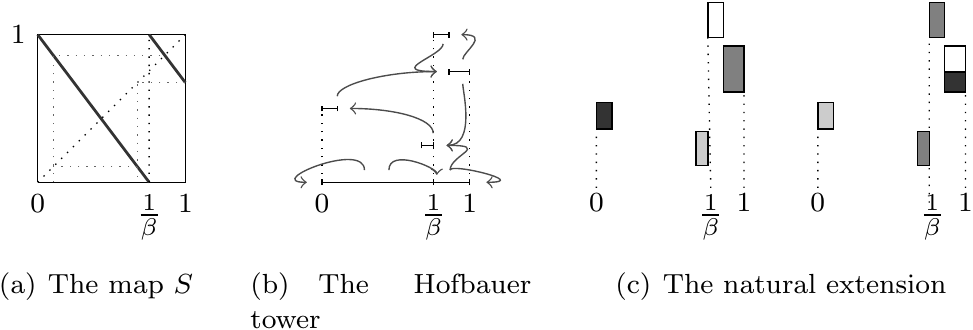}
\caption{The negative $\beta$-transformation with $\beta$ equal to the real root of $x^3-x-1$ and its natural extension.}
\label{f:negbeta}
\end{figure}

Note that in general $\mu_2$ is not necessarily equivalent to Lebesgue on the unit interval, see Figure~\ref{f:negbeta} for a specific example. In \cite{LS11} Liao and Steiner show that $S$ is exact, hence by results from Rohlin \cite{Roh64}, the natural extension is a $K$-automorphism. Now the previously mentioned result from Rychlik \cite{Ryc83} gives that the natural extension of $S$ is weakly Bernoulli and thus so is $S$ itself.

\vskip .2cm
Now suppose that $1$ has an eventually periodic orbit for both $T$ and $S$. This happens for example when $\beta$ is a {\em Pisot number}, \ie a real-valued algebraic integer larger than $1$ with all its Galois conjugates in modulus less than $1$ (see \cite{Sch80} for $T$ and \cite{FL09} for $S$). Then the natural extensions of both maps contain only finitely many rugs and the Markov shifts constructed in Section~\ref{ss:natural_extension} have finite alphabets. Let $(Y_T, \mathcal C_T, \nu_T, F_T)$ and $(Y_S, \mathcal C_S, \nu_S, F_S)$ denote the natural extensions of $T$ and $S$. Since both $T$ and $S$ have the same entropy and are exact and thus strongly mixing, also $F_T$ and $F_S$ have the same entropy and are strongly mixing. The results from \cite{KS79} by Keane and Smorodinsky show that then the natural extensions $T$ and $S$ are {\em finitarily isomorphic}, \ie there is an a.e.~continuous isomorphism from $Y_T$ to $Y_S$. 

\section{Further examples}\label{sec:examples}

\subsection{Higher integer dimensions}
Let $T:[0,1]^d \to [0,1]^d$ be a piecewise affine expanding map of the form $T x = A x$ (mod $\Z^d$). Let $Z_1, \ldots, Z_N$ be the pieces on which $T$ is continuous. Since $T$ is expanding in all directions, we get (c1) and (c3). The ergodicity from (c5) is clear for Lebesgure measure $m^d$ and since $T$ is piecewise affine, (c2) also holds. For each $1 \le j \le N$ and each measurable set $E \subseteq Z_j$ we have $m^d (TE) = |\text{det}(A)| \ m^d(E)$, which gives (c4). Condition (c6) follows from the fact that $m^d\big( \cup_{j=1}^N \partial Z_j \big)=0$ combined with (c2) and (c4). Hence, $m^d$ satisfies conditions (c1)-(c6) and is not $T$-invariant. Then by Proposition~\ref{prop:noninvariant}, $m^d$ is liftable and the construction from Section~\ref{ss:natural_extension} gives the natural extension.

\subsubsection{Random $\beta$-transformation}\label{ss:randombeta}
One specific example of a piecewise affine conformal map we give here is a variation of the random $\beta$-transformation, which was first introduced in \cite{DK03}. If $1< \beta <2$, then almost every point has infinitely many different number expansions of the form $\sum_{k=1}^{\infty} \frac{b_k}{\beta^k}$, where $b_k \in \{0,1\}$. The random $\beta$-transformation gives for each point all possible such expansions in base $\beta$ and is basically defined as the product of an independent coin tossing process and two isomorphic copies of the map $x \mapsto \beta x \pmod 1$ on an extended interval. Consider the space $X=[0,1] \times \big[ 0, \frac{1}{\beta-1} \big]$, with the partition $\mathcal Z = \{ Z_j \}_{j=1}^6$ given by
\[\begin{array}{ccc}
Z_1 = \big[ 0,\frac12 \big) \times \big[ 0,\frac{1}{\beta}\big), & Z_2 = \big[ 0,\frac{1}{2}\big) \times \big[ \frac{1}{\beta}, \frac{1}{\beta(\beta-1)}\big], & Z_3 = \big[0,\frac12\big) \times \big( \frac{1}{\beta (\beta-1)}, \frac{1}{\beta-1} \big].\\
\\
Z_4 = \big[ \frac12,1\big] \times \big[ 0,\frac{1}{\beta}\big), & Z_5 = \big[ \frac{1}{2},1\big] \times \big[ \frac{1}{\beta}, \frac{1}{\beta(\beta-1)}\big], & Z_6 =\big[\frac12,1\big] \times \big( \frac{1}{\beta (\beta-1)}, \frac{1}{\beta-1} \big].
\end{array}\]
The transformation $T:X \to X$ is defined by
\[ T(x,y) = \left\{
\begin{array}{ll}
(2x \, (\text{mod }1), \beta y), & \text{if } (x,y) \in Z_1 \cup Z_2 \cup Z_4,\\
\\
(2x \, (\text{mod }1), \beta y -1), & \text{if } (x,y) \in Z_3 \cup Z_5 \cup Z_6.
\end{array}
\right.\]
The reasons why (c1)-(c6) hold for $m^d$ are the same as in the previous example. Proposition~\ref{prop:noninvariant} gives that $m^d$ is liftable to a measure $\hat \mu$ on the Hofbauer tower. Hence, we can construct the natural extension of $T$ as outlined in Section~\ref{ss:natural_extension}.

Originally the random $\beta$-transformation is not defined as a proper skew product, see \cite{DK03}. Instead of always applying the doubling map in the second coordinate, they only apply the doubling map in the middle region. Below we give a specific example of the random $\beta$-transformation defined in this way and construct the natural extension for this value of $\beta$. Let $\beta = \frac{1+\sqrt 5}{2}$ be the golden ratio and define the map $K:[0,1] \times  [0, \beta]  \to [0,1] \times  [0, \beta]$ by
\[ K(x,y) = \left\{
\begin{array}{ll}
(y, \beta x), & \text{if }  x < 1/\beta,\\
(2y, \beta x), & \text{if } (x,y) \in [1/\beta,1] \times [0, 1/2),\\
(2y-1, \beta x -1), & \text{if } (x,y) \in [1/\beta,1] \times [1/2,1],\\
(y, \beta x -1), & \text{if } x>1.
\end{array} \right.\] 
Then
\[ \begin{array}{lll}
Z_1 = [0,1] \times [0, 1/\beta] ,  && Z_2 = [0,1/2] \times [1/\beta,1],\\
Z_3 = [1/2,1] \times [1/\beta,1], && Z_4 = [0,1] \times [1,\beta].
\end{array}\]
Since this already is a Markov partition, $\mathcal D = \{Z_1, Z_2, Z_3, Z_4\}$. Using the formula from (\ref{q:pwdensity}) we get that for each $n \ge 2$,
\[ \rho_n(Z_1) = \frac1n \Big(1 +  \sum_{k=2}^n \frac{f_{k+2}}{2\beta^k} \Big), \]
where $f_j$ is the $j$-th element in the Fibonacci sequence starting with $f_1=1$, $f_2=1$. Using the direct formula for the elements in the Fibonacci sequence gives
\[ \rho(Z_1)= \lim_{n \to \infty} \rho_n(Z_1) = \frac{\beta^3}{2(\beta^2+1)}.\]
By symmetry we get the same value for $Z_4$. Since $\rho(Z_2) = \frac{\rho(Z_1)}{\beta} + \frac{\rho(Z_4)}{\beta}$, we have $\rho(Z_2)=\rho(Z_3) = \frac{\beta^2}{\beta^2+1}$. In Figure~\ref{f:random} we see the natural extension for this random $\beta$-transformation $K$.
\begin{figure}[ht]
\centering
\includegraphics{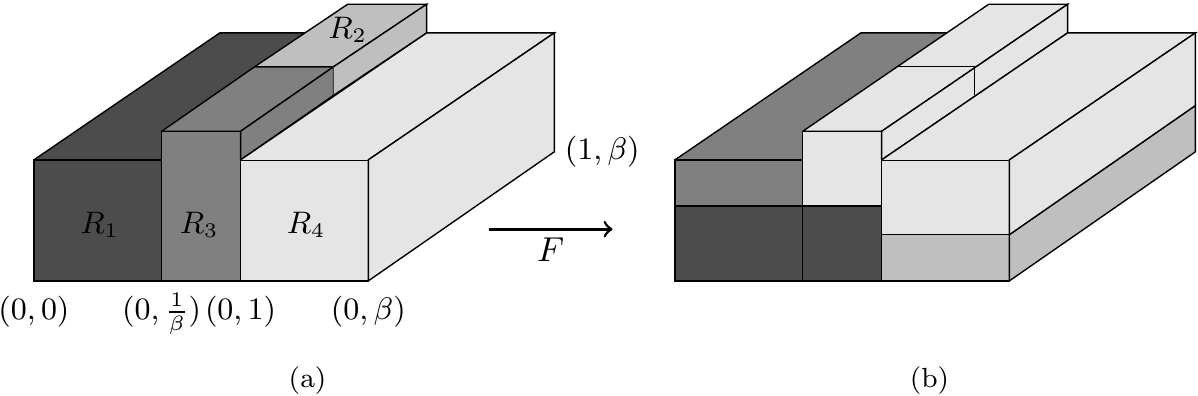}
\caption{The natural extension for the random $\beta$-transformation with $\beta$ equal to the golden ratio. $F$ maps the regions in (a) to the regions in (b) with the same colour.}
\label{f:random}
\end{figure}

\subsection{Balanced measures}\label{ss:balanced}
We say that $T:X \to X$ is $d$-to-$1$ if there is a partition $\{ Z_j \}_{j=1}^d$ of $X$, generating the $\sigma$-algebra of measurable sets, such that $T:Z_j \to X$ is a measurable bijection and $Z_i \cap Z_j$ is negligible (e.g.~countable or of measure zero w.r.t.\ the measure used). A measure $\mu$ is {\em balanced} if the Jacobian $J(x) \equiv d$. In this case, $(\mathcal D, \to)$ is the full graph on $\{ Z_1, \ldots, Z_d\}$, so conditions (c4) and (c6) are trivially satisfied. Therefore we can construct the natural extension by the method of Section~\ref{ss:natural_extension} if the system satisfies (c1), (c2) and (c3) and if the measure is ergodic. We give two examples.

\subsubsection{Rational maps on the Julia set}\label{sss:rationalmaps}
Let $R:\hat \C \to \hat \C$ be a rational map of degree $d \geq 2$ on the Riemann sphere, \ie $R(z) = \frac{P(z)}{Q(z)}$ where $P$ and $Q$ are two polynomials with no common factor and $d = \max\{ \deg P, \deg Q\}$. When restricted to the Julia set, one can find a generating partition $\{ Z_j\}$ w.r.t.~which $R$ is $d$-to-$1$ giving (c1) and (c3). This goes back to Ma\~n\'e \cite{Man83} and the corresponding balanced measure is well-defined (\ie independent of the choice of $\{ Z_j\}$) as well as the unique invariant measure of maximal entropy. This gives (c2). Following conjectures and partial results by Ma\~n\'e \cite{Man85} and Lyubich \cite{Lyu83}, and using techniques of Hoffman and Rudolph \cite{HR02} it was shown that $\mu$ is isomorphic to the $(1/d, \dots, 1/d)$ one-sided Bernoulli shift, see \cite{HH02}, and hence $\mu$ is ergodic.
Explicit construction for a Bernoulli partitions (for Latt\`es examples)
can be found in \cite{BK00,Kos02}.

\subsubsection{Certain endomorphisms on the torus}\label{sss:torus}
Let $X = \mathbb T^n = \R^n / \Z^n$ be the $n$-dimensional torus, and $T$ an endomorphism of the form $T(x) = h(Ax) \pmod{\Z^n}$, where $h:\mathbb T^n \to \mathbb T^n$ is a homeomorphism homotopic to the identity, and $A$ an $n \times n$ integer matrix with $\det(A) = \pm d$. If $h$ is the identity, 
then Lebesgue measure is a balanced measure, see \cite{DH93} for
some intricancies of its natural extensions and factor spaces. 
A priori, a $d$-to-$1$ partition $\{ Z_j \}_{j=1}^d$ need not be unique; more importantly, it is not automatically generating. For example, if
\[ A = \begin{pmatrix} 6 & 4 \\ 2 & 2 \end{pmatrix} =
\begin{pmatrix} 3 & 2 \\ 1 & 1 \end{pmatrix} \cdot
\begin{pmatrix} 2 & 0 \\ 0 & 2 \end{pmatrix} \]
with eigenvalues $\lambda_\pm = 4 \pm 2\sqrt{3}$, then the eigenspace of the second eigenvalue represents a contracting direction, and for this reason the partition of $\mathbb T^2$ in, say, four quarters, is not generating in forward time. In fact, there exists no forward time generating $4$-to-$1$ partition because the topological entropy is $\log(4 + 2\sqrt{3}) > \log 4$. See Kowalski \cite{Kow88} for some interesting results in this direction.

\noindent
Henk Bruin\\
Fakult\"at f\"ur Mathematik \\ 
Universit\"at Wien \\ 
Nordbergstra{\ss}e 15/Oskar Morgensternplatz 1, A-1090 Wien\\
Austria\\
henk.bruin@univie.ac.at
\\[5mm]
Charlene Kalle\\
Mathematisch Instituut\\ 
Leiden University\\ 
Niels Bohrweg 1, 2333CA Leiden\\ 
The Netherlands\\
kallecccj@math.leidenuniv.nl

\end{document}